\documentclass[12pt]{amsart}

\usepackage{amssymb}
\usepackage{enumerate}
\usepackage{hyperref}
\usepackage{amsmath,amssymb}
\newtheorem{theorem}{Theorem}[section]

\newtheorem{lemma}[theorem]{Lemma}
\newtheorem{corollary}[theorem]{Corollary}
\theoremstyle{definition}
\newtheorem{definition}[theorem]{Definition}

\theoremstyle{remark}
\newtheorem{remark}[theorem]{Remark}

\numberwithin{equation}{section}

\begin{document}

\title[Isocapacity Estimates for Hessian Operators]
{Isocapacity Estimates for Hessian Operators}


\author[Jie Xiao]{Jie Xiao}
\address{Jie Xiao\\Department of Mathematics and Statistics, Memorial University, St. John's, NL A1C 5S7, Canada}
\email{jxiao@mun.ca}
\author[Ning Zhang]{Ning Zhang}
\address{Ning Zhang\\Department of Mathematical and Statistical Sciences, University of Alberta, Edmondon, T6G 2G1, Canada}
\email{nzhang2@ualberta.ca}

\thanks{Project supported by NSERC of Canada as well as by URP of Memorial University, Canada.}

\date{}


\begin{abstract} Through a new powerful potential-theoretic analysis, this paper is devoted to discovering the geometrically equivalent isocapacity forms of Chou-Wang's Sobolev type inequality and Tian-Wang's Moser-Trudinger type inequality for the fully nonlinear $1\le k\le {n}/2$ Hessian operators.
\end{abstract}

\maketitle

\tableofcontents 


\section{Hessian Sobolev through isocapacitary inequalities}\label{s1} 

\subsection{Sobolev type inequalities for Hessian operators}\label{s11}

Unless a special remark is made, from now on, $\Omega$ is a bounded smooth domain in the $n$-dimensional Euclidean space $\mathbb R^n$ with $n\geq 2$. Let $u$ be a $C^2$ real-valued function on $\Omega$. For each integer $k\in[1,n]$, the $k$-Hessian operator $F_k$ is defined as
\begin{equation*}\label{e11}
F_k[u]=S_k(\lambda(D^2u))=\sum_{1\leq i_1< \dots < i_k\leq n}\lambda_{i_1}\cdots \lambda_{i_k},
\end{equation*} 
where $\lambda=(\lambda_1,\dots,\lambda_n)$ is the vector of the eigenvalues of the real symmetric Hessian matrix $[D^2u]$. In particular, one has:
$$
F_k[u]=\begin{cases}
\Delta u=\hbox{the\ Laplace\ operator}\ \hbox{ for } k=1;\\
\hbox{a\ fully\ nonlinear\ operator}\ \hbox{ for } 1<k<n;\\
\hbox{det}(D^2u)=\hbox{the\ Monge-Amp\'ere\ operator}\ \hbox{ for } k=n.
\end{cases}
$$ 
Here and henceforth, the following facts should be kept in mind: for $1<k<n$ each $F_k[u]$ is degenerate elliptic for any $k$-convex or $k$-admissible function $u$, denoted by $u\in\Phi^k(\Omega)$, namely, any $C^2(\Omega)$ function $u$ enjoying 
$$
F_j[u]\ge 0\ \ \hbox{on}\ \ \Omega\ \ \forall\ \ j=1,2,...,k. 
$$
Moreover, if $\Phi_0^k(\Omega)$ stands for the class of all functions $u\in\Phi^k(\Omega)$ with zero value on the boundary $\partial\Omega$ of $\Omega$, then $\Phi_0^k(\Omega)\not=\emptyset$ amounts to that $\partial\Omega$ is $(k-1)$-convex, i.e., the $j$-th mean curvature 
$$
H_{j}(\partial\Omega,x)=\frac{\sum_{1\le i_1<...<i_{j}\le n-1}\kappa_{i_1}(x)\cdots\kappa_{i_j}(x)}{\Big(\begin{array}{c} n-1\\ j\end{array}\Big)}\ \ \forall\ \ j=1,...,k-1
$$
of the boundary $\partial\Omega$ at $x$ is nonnegative, where $\kappa_1(x),...,\kappa_{n-1}(x)$ are the principal curvatures of $\partial\Omega$ at the point $x$; see for example \cite{CaNiSp, Fer, Ga, Lab,  TiW, TW2,  W2}.

As a natural generalization of the well-known case $k=1$, the following Sobolev type inequalities indicate that $\Phi_0^k$ can be embeded into some integrable function spaces; see Wang \cite{W1}, Chou \cite{Ch1, Ch2}, and Tian-Wang \cite{TiW} for the details.

\begin{theorem}\label{t11} Let 
$$
\begin{cases}
1\le k\le {n};\\
u\in\Phi_0^k(\Omega);\\
\|u\|_{\Phi_0^k(\Omega)}=\Big(\int_{\Omega} (-u)F_k[u] \Big)^{1/(k+1)}.
\end{cases}
$$
\begin{enumerate}[\rm(i)]
\item If $1\leq k<\frac{n}{2}$ and $1\le q\le k^*=\frac{n(k+1)}{n-2k}$, then there is a positive constant $c(n,k,q,|\Omega|)$ depending only on $n$, $k$, $q$, and the volume $|\Omega|$ of $\Omega$ such that the Sobolev type inequality
	\begin{equation*}\label{e14}
	\|u\|_{L^{q}(\Omega)}\leq c(n,k,q,|\Omega|)\|u\|_{\Phi_0^k(\Omega)}
	\end{equation*}
holds, where for $q=k^*$ the best constant in the last estimate is obtained via letting $u: \Omega\to\mathbb{R}^n$ be
	\begin{equation*}\label{e15}
	u(x)=\big(1+|x|^2 \big)^{\frac{2k-n}{2k}}.
	\end{equation*}
Moreover, for $k=\frac{n}{2}$ and $0<q<\infty$, there is a positive constant $c(n,k,q,\hbox{diam}(\Omega))$ depending only on
$n,k,q$ and the diameter $\hbox{diam}(\Omega)$ of $\Omega$ such that the Sobolev type inequality
 \begin{equation*}\label{e16}
	\|u\|_{L^q(\Omega)}\leq c(n,k,q,\textnormal{diam}(\Omega))\|u\|_{\Phi_0^k(\Omega)}
	\end{equation*}
	holds.
	
\item If $k=\frac{n}{2}$, then there is a positive constant $c(n,\hbox{diam}(\Omega))$ depending only on $n$ and $\hbox{diam}(\Omega)$ such that the Moser-Trudinger type inequality
	\begin{equation*}
	\label{e16e}
\sup_{0<\|u\|_{\Phi_0^k(\Omega)}<\infty}\int_\Omega\exp\left(\alpha\Big(\frac{|u|}{\|u\|_{\Phi_0^k(\Omega)}}\Big)^\beta\right)\le c(n,\hbox{diam}(\Omega))
	\end{equation*}
	holds, where 
	$$
	\begin{cases}
	0<\alpha\le\alpha_0=n\left(\frac{\omega_n}{k}\tiny{\big(\begin{array}{c} n-1\\ k-1\end{array}\big)}\right)^\frac2n;\\
	1\le\beta\le\beta_0=1+\frac2n;\\
	\omega_n=\hbox{the\ surface\ area\ of\ the\ unit\ sphere\ in}\ \mathbb R^{n}.
	\end{cases}
	$$
	
\item If $\frac{n}{2}< k\leq n$, then there is a positive constant $c(n,k,\textnormal{diam}(\Omega))$ depending only on $n,k$ and $\hbox{diam}(\Omega)$ such that the Morrey-Sobolev type inequality
	\begin{equation*}\label{e17}
	\|u\|_{L^{\infty}(\Omega)}\leq c(n,k,\textnormal{diam}(\Omega))\|u\|_{\Phi_0^k(\Omega)}
	\end{equation*}
	holds.

\end{enumerate}
\end{theorem}

\subsection{Statement of Theorems \ref{t12}-\ref{t13}}\label{s12} Since the Morrey-Sobolev type inequality in Theorem \ref{t11} (iii) is relatively independent (cf. \cite{Tal}), a natural question comes up: {\it what is the geometrically equivalent form of Theorem \ref{t11} (i)-(ii)?} To answer this question, we need the so-called $k$-Hessian capacity that was introduced by Trudinger-Wang \cite{TW3} in a way similar to the capacity defined by Bedford-Taylor in \cite{BT} for the purisubharmonic functions. To be more precise, if $K$ is a compact subset of $\Omega$, then the $[1,n]\ni k$ Hessian capacity of $K$ with respect to $\Omega$ is determined by
\begin{equation*}\label{e18}
cap_k(K,\Omega)=\sup\left\{\int_K F_k[u]:\ u\in \Phi^k(\Omega),\ -1<u<0\right\};
\end{equation*}
and hence for an open set $O\subset \Omega$ we define
\begin{equation*}\label{e19}
cap_k(O,\Omega)=\sup\Big\{cap_k(K,\Omega):\ \textnormal{compact}\ K\subset O\Big\};
\end{equation*}
whence giving the definition of $cap_k(E,\Omega)$ for an arbitrary set $E\subset \Omega$: 
\begin{equation*}\label{e110}
cap_k(E,\Omega)=\inf\Big\{cap_k(O,\Omega):\ \textnormal{open}\ O\ \hbox{with}\ E\subset O\subset\Omega\Big\}.
\end{equation*}
Remarkably, Phuc-Verbitsky's \cite[Theorem 2.20]{PV} reveals that if $1\le k<n/2$ and $K$ is a compact subset of $\overline{Q}$ with $Q$ being any element of a Whitney decomposition of $\Omega$ into a union of disjoint dyadic cubes then $cap_k(K,\Omega)$ is equivalent to the so-called Bessel capacity $Cap_{{\mathbf G}_{{2k}/{(k+1)}},k+1}(K,\Omega)$.  

According to Labutin's computation in \cite[(4.16)-(4.17)]{Lab}, we see that if $B_\rho\subset\mathbb{R}^n$ denotes the open ball centered at the origin with radius $\rho>0$ and if $0<r<R<\infty$ then there is a constant $c(n,k)>0$ depending only on $n,k$ such that
$$
cap_k(B_r,B_R)=
\begin{cases}
c(n,k)\Big(r^{2-\frac{n}{k}}-R^{2-\frac{n}{k}}\Big)^{-k}\textnormal{ for } 1\leq k< \frac{n}{2};\\
c(n,k)\Big(\log\frac{R}{r}\Big)^{\frac{n}{2}}\textnormal{ for }  k=\frac{n}{2}.
\end{cases}
$$
Moreover, $cap_k(\cdot,\Omega)$ has the following metric properties (cf. \cite[Lemma 4.1]{Lab}):

\begin{itemize}
\item if $E=\emptyset$, then $cap_k(E,\Omega)=0$;

\item if $E_1\subset E_2\subset\Omega$, then $cap_k(E_1,\Omega)\le cap_k(E_2,\Omega)$;

\item if $E\subset\Omega_1\subset\Omega_2$, then $cap_k(E,\Omega_1)\ge cap_k(E,\Omega_2)$;

\item if $E_1,E_2,\cdots\subset\Omega$, then $cap_k(\cup_j E_j,\Omega)\le\sum_{j}cap_k(E_j,\Omega)$;

\item if $K_1\supset K_2\supset\cdots$ is a sequence of compact subsets of $\Omega=B_R$, then $cap_k(\cap_{j} K_j,\Omega)=\lim_{j\to\infty}cap_k(K_j,\Omega)$.
\end{itemize}

As a geometric form of Theorem \ref{t11} (i)-(ii), we have the following isocapacitary inequalities for the $k$-Hessian operators -- see also Maz'ya \cite[(8.8)-(8.9)]{Maz} for the case $k=1$.
 
\begin{theorem}\label{t12} Let $E\subset \Omega$ and $1\le k\le\frac{n}{2}$. 
\begin{enumerate}[\rm(i)]

\item If $1\leq k<\frac{n}{2}$ and $1\leq q\leq \frac{n(k+1)}{n-2k}$, then there exists a constant $c(n,k,q,|\Omega|)>0$ depending only on $n,k,q$, and $|\Omega|$ such that
\begin{equation*}\label{e111}
|E|^\frac{k+1}{q}\leq c(n,k,q,|\Omega|)cap_k(E, \Omega).
\end{equation*}
In particular, for $q=\frac{n(k+1)}{n-2k}$ there exists a constant $c(n,k)>0$ depending only on $n,k$ such that
\begin{equation*}\label{e1111}
|E|^\frac{n-2k}{n}\leq c(n,k)cap_k(E, \Omega).
\end{equation*}
Moreover, for $k=\frac{n}{2}$ and $1\leq q<\infty$ there is a positive constant $c(n,k,q,\textnormal{diam}(\Omega))$ depending only on $n,k,q$ and $\textnormal{diam}(\Omega)$ such that
\begin{equation*}\label{e113}
|E|^\frac{k+1}{q}\leq c(n,k,q,\textnormal{diam}(\Omega))cap_k(E,\Omega).
\end{equation*}
\item If $k=\frac{n}{2}$, then there is a constant $c(n)>0$ depending only on $n$ such that
\begin{equation*}\label{e112}
\frac{|E|}{|\Omega|}\le c(n)\exp\left(-\frac{\alpha}{\big(cap_k(E,\Omega)\big)^{\frac{\beta}{k+1}}}\right)
\end{equation*}
holds for a constant $c(n,k)$ depending only on $n,k$, where 
    $$
	\begin{cases}
	0<\alpha\le\alpha_0=n\left(\frac{\omega_n}{k}\tiny{\big(\begin{array}{c} n-1\\ k-1\end{array}\big)}\right)^\frac2n;\\
	1\le\beta\le\beta_0=1+\frac2n;\\
	\omega_n=\hbox{the\ surface\ area\ of\ the\ unit\ sphere\ in}\ \mathbb R^{n}.
	\end{cases}
	$$
\end{enumerate}
\end{theorem}

Theorem \ref{t12} (i)-(ii) will be verified in \S \ref{s3} by using Theorem \ref{t11} (i)-(ii) and a new characterization of $cap_k(\cdot,\Omega)$ given in \S \ref{s2}. This process indicates Theorem \ref{t11} (i)-(ii) $\Rightarrow$ Theorem \ref{t12} (i)-(ii). Nevertheless, Theorem \ref{t13} (i)-(ii) below with $\mu$ being the $n$-dimensional Lebesgue measure shows Theorem \ref{t12} (i)-(ii) $\Rightarrow$ Theorem \ref{t11} (i)-(ii) under $\Omega$ being an origin-centered ball and $k+1\leq q\leq \frac{n(k+1)}{n-2k}$.

\begin{theorem}\label{t13} Given an origin-centered Euclidean ball $\Omega\subset\mathbb R^n$, $1\le k\le \frac{n}{2}$, and a nonnegative Randon measure $\mu$ on $\Omega$, let
$$
\tau(\mu,\Omega,t)=\inf\big\{cap_k(K,\Omega):\ \hbox{compact}\ K\subset\Omega\ \hbox{with}\ \mu(K)\ge t\big\}\ \forall\ t>0.
$$
be the $k$-Hessian capacitary minimizing function with respect to $\mu$.

\begin{enumerate}[\rm(i)]
\item If $1\le k\le \frac{n}{2}$, then
\begin{equation*}
\label{e1e}
\sup\left\{\frac{\|u\|_{L^q(\Omega,\mu)}}{\|u\|_{\Phi^k_0(\Omega)}}:\  u\in \Phi_0^k(\Omega)\cap C^2(\bar{\Omega}),\ \ 0<\|u\|_{\Phi^k_0(\Omega)}<\infty\right\}<\infty
\end{equation*}
holds if and only if
$$
\begin{cases}
\label{e1e1}
\sup_{t>0} \frac{t^\frac{k+1}{q}}{\tau(\mu,\Omega,t)}<\infty\ \hbox{ for }\ k+1\leq q<\infty;\\
\label{e2e2}
\int_0^\infty \Big(\frac{t^\frac{k+1}{q}}{\tau(\mu,\Omega,t)}\Big)^\frac{q}{k+1-q}\,\frac{dt}{t}<\infty\ \hbox{ for }\ 1<q<k+1.
\end{cases}
$$
\item If $k=\frac{n}{2}$, then
\begin{equation*}
\label{e1e3}
\sup\Big\{\|u\|_{L^1_\varphi(\Omega,\mu)}: \  u\in \Phi_0^k(\Omega)\cap C^2(\bar{\Omega}),\ \ 0<\|u\|_{\Phi^k_0(\Omega)}<\infty\Big\}<\infty
\end{equation*}
holds if and only if
$$
\label{e1e4}
\sup_{t>0} t\exp\left(\frac{\alpha}{\big(\tau(\mu,\Omega,t)\big)^{\frac{\beta}{k+1}}}\right)<\infty,
$$
where 
$$
\begin{cases}
\|u\|_{L^1_\varphi(\Omega,\mu)}=\int_\Omega\varphi(u)\,d\mu;\\
\varphi(u)=\exp\left(\alpha\Big(\frac{|u|}{\|u\|_{\Phi_0^k(\Omega)}}\Big)^\beta\right);\\
0<\alpha<\alpha_0=n\left(\frac{\omega_n}{k}\tiny{\big(\begin{array}{c} n-1\\ k-1\end{array}\big)}\right)^\frac2n;\\
	1\le\beta\le\beta_0=1+\frac2n;\\
	\omega_n=\hbox{the\ surface\ area\ of\ the\ unit\ sphere\ in}\ \mathbb R^{n}.
\end{cases}
$$
\end{enumerate}
\end{theorem} 

Often referred to as trace estimates (due to the fact that $\mu$ lives on $\Omega$ and may be the surface measure on a smooth submanifold of $\Omega$), the results in Theorem \ref{t13} will be proved in \S \ref{s5} through the $k$-Hessian capacitary weak and strong type estimates for $\|\cdot\|_{\Phi^k_0(\Omega)}$ presented in \S \ref{s4}, paralleling the classic ones for $k=1$ as developed in Adams-Hedberg's \cite[Chapter 7]{AH}. Here, it is worth pointing out that the case $k=1$ of Theorem \ref{t13} can be read off from the case $p=2$ of Maz'ya's \cite[Theorem 8.5 \& Remark 8.7]{Maz} (related to the Nirenberg-Sobolev inequality \cite[Lemma VI.3.1]{Cha}), and the case $q=k+1$ of Theorem \ref{t13} leads to a kind of Cheeger's inequality - for $k=1$ see also \cite{Che}, \cite[Theorem VI.1.2]{Cha}, and \cite{X1}.

\begin{remark}\label{r11} Two more comments are in order:

\begin{enumerate}[\rm(i)]
\item Upon adapting the relatively natural capacity of a compact $K\subset\Omega$ for $k$-Hessian operator below (cf. \S \ref{s2})
$$
cap_{k,3}(K,\Omega)=\inf\Big\{\|u\|^{k+1}_{\Phi^k_0(\Omega)}:\ u\in \Phi_0^k(\Omega)\cap C^2(\bar{\Omega}),\ u|_K\le -1,\ u\le 0\Big\},
$$
we can see that Theorem \ref{t13} without assuming that $\Omega$ is an origin-centerd Euclidean ball, still hold with $cap_k(\cdot,\Omega)$ being replaced by $cap_{k,3}(\cdot,\Omega)$.

\item While going along with demonstrating Theorems \ref{t12}-\ref{t13}, we will introduce the required notation. But here, we only write $c(a,b,c,d)$ for different constants (in different lines) depending only on $a$, $b$, $c$, $d$ - for instance - $X\leq c(a,b,c,d)Y\leq c(a,b,c,d)Z$ means that there exist two positive constants $c(a,b,c,d)$ and $\tilde{c}(a,b,c,d)$ depending only on $a$, $b$, $c$, $d$ such that $X\leq c(a,b,c,d)Y\leq \tilde{c}(a,b,c,d)Z$.

\end{enumerate}

\end{remark}
 
\section{Four alternatives to $cap_k(\cdot,\Omega)$}\label{s2} 


The purpose of this section is to define four new types of the $k$-Hessian capacity with $1\leq k\leq \frac{n}{2}$ and then to establish their relations with $cap_k(\cdot,\Omega)$. The forthcoming approach is entirely different from the Wolff's potential-based argument for Phuc-Verbitsky's \cite[Theorem 2.20]{PV} - the comparability between both $cap_k(\cdot,\Omega)$ and $Cap_{{\mathbf G}_{{2k}/{(k+1)}},k+1}(\cdot,\Omega)$ as $1\le k<n/2$.

\begin{definition}\label{d21} Suppose 
$$ 
\begin{cases}
1\leq k\leq \frac{n}{2};\\
1_E\ \hbox{stands\ for\ the\ characteristic\ function\ of}\ E\subset\Omega.
\end{cases}
$$ 
\begin{enumerate}[\rm(i)]

\item For a compact $K\subset \Omega$ let 
$$
\begin{cases}
cap_{k,1}(K,\Omega)=\sup\Big\{\int_K F_k[u]:\ u\in \Phi_0^k(\Omega)\cap C^2(\bar{\Omega}),\ -1<u<0\Big\};\\
cap_{k,2}(K,\Omega)=\inf\Big\{\int_{\Omega} F_k[u]:\ u\in \Phi_0^k(\Omega)\cap C^2(\bar{\Omega}),\ u\leq -1_K\Big\};\\
cap_{k,3}(K,\Omega)=\inf\Big\{-\int_{\Omega} uF_k[u]:\ u\in \Phi_0^k(\Omega)\cap C^2(\bar{\Omega}),\ u\leq -1_K\Big\};\\
cap_{k,4}(K,\Omega)=\sup\Big\{-\int_K uF_k[u]:\ u\in \Phi_0^k(\Omega)\cap C^2(\bar{\Omega}), -1<u<0\Big\}.
\end{cases}
$$
\item For an open set $O\subset\Omega$ and $j=1,2,3,4$ set
$$
cap_{k,j}(O,\Omega)=\sup\Big\{cap_{k,j}(K,\Omega):\ \textsl{compact}\ K\subset O\Big\}.
$$
\item For a general set $E\subset\Omega$ and $j=1,2,3,4$ put
$$
cap_{k,j}(E,\Omega)=\inf\Big\{cap_{k,j}(K,\Omega):\ \textsl{open}\ O\ \hbox{with}\ E\subset O\subset\Omega\Big\}.
$$
\end{enumerate}
\end{definition}

\begin{lemma}\label{l21} Suppose $1\leq k\leq \frac{n}{2}$. Let $\Omega$ be the Euclidean ball $B_r$ of radius $r$ centered at the origin. If $K$ is a compact subset of $\Omega$, then 
$$
cap_{k,j}(K,\Omega)=
\begin{cases}
\int_K F_k[R_k(K,\Omega)]\ \textsl{ for }\ j=1;\\
 \int_K (-R_k(K,\Omega))F_k[R_k(K,\Omega)]\ \textsl{ for }\ j=4,
\end{cases}
$$
where
\begin{equation*}\label{e21}
R_k(K,\Omega)(x)=\limsup_{y\to x}\Big(\sup\big\{ u(y): u\in\Phi_0^k(\Omega),\ u\leq-1_K \big\}\Big)
\end{equation*}
is the regularised relative extremal function associated with $K\subset\Omega$.
\end{lemma}
\begin{proof} As showed in \cite{Lab}, the function $x\mapsto R_k(K,\Omega)(x)$ is upper semicontinuous; is of $C^2(\bar{\Omega})$; and is the viscosity solution of the following Dirichlet problem:
\begin{equation*}\label{e22}
\begin{cases}
F_k[u]=0 &\mbox{in } \Omega \backslash K; \\
u=-1 & \mbox{on } \partial K;\\
u=0  & \mbox{on } \partial \Omega.
\end{cases}
\end{equation*}
Moreover, 
\begin{equation*}\label{e23}
cap_k(K,\Omega)= \int_K F_k[R_k(K,\Omega)].
\end{equation*}
Note that $R_k(K,\Omega)$ is in $\Phi_0^k(\Omega)\cap C^2(\bar{\Omega})\subset \Phi^k(\Omega)$. So, from Definition \ref{d21} it follows that
\begin{equation*}\label{e25}
cap_{k,1}(K,\Omega)=\int_K F_k[R_k(K,\Omega)].
\end{equation*}

To see the desired formula for $j=4$, let $u\in \Phi_0^k(\Omega)\cap C^2(\bar{\Omega})$. Then, for any $\epsilon>0$ there exists a function $v\in \Phi_0^k(\Omega)\cap C^2(\bar{\Omega})$ such that $v=(1+\epsilon)u$ and
\begin{eqnarray*}
&&(1+\epsilon)^{k+1}cap_{k,4}(K,\Omega)\\
&&=(1+\epsilon)^{k+1}\sup\Big\{\int_K (-u)F_k[u]:\ u\in \Phi_0^k(\Omega)\cap C^2(\bar{\Omega}),\ -1<u<0\Big\}\\
&&=\sup\Big\{\int_K (-v)F_k[v]:\ v\in \Phi_0^k(\Omega)\cap C^2(\bar{\Omega}),\ -1-\epsilon<v<0\Big\}.
\end{eqnarray*}
Since $R_k(K,\Omega)>-1-\epsilon$ in $K$, it follows that
$$
(1+\epsilon)^{-(k+1)}\int_K (-R_k(K,\Omega))F_k[R_k(K,\Omega)]\leq cap_{k,4}(K,\Omega).
$$
Letting $\epsilon\to 0$, we obtain
$$
\int_K (-R_k(K,\Omega))F_k[R_k(K,\Omega)]\leq cap_{k,4}(K,\Omega).
$$

To reach the reversed form of the last inequality, let $\{K_i\}$ be a sequence of compact subsets of $\Omega$ with smooth boundary and 
$$
K_{i+1}\subset K_i\quad\&\quad\cap_{i=1}^{\infty}K_i=K.
$$
Then, using the regularity of $\partial K_i$ we define
$$
u_i=R_k(K_i,\Omega)\in C(\bar{\Omega}).
$$
According to \cite[Lemma 2.1]{TW1}, we have the following monotonicity: if
$$
\begin{cases}
u,v\in \Phi^k(\Omega)\cap C^2(\bar{\Omega});\\
u\geq v\quad\hbox{in}\quad\Omega;\\
u=v\quad\hbox{on}\quad\partial \Omega,
\end{cases}
$$
then
\begin{equation*}\label{e27}
\int_{\Omega}F_k[u]\leq \int_{\Omega}F_k[v],
\end{equation*}
whence getting via \cite[(4.9)]{Lab}
$$ 
\int_{K} F_k[u]\le\int_{\{u_i\le u\}} F_k[u]\le \int_{\Omega}F_k[u]\le \int_{\Omega}F_k[u_i]=\int_{K_i}F_k[u_i].
$$
Letting $i\to \infty$ in the last inequality yields that
\begin{equation*}\label{e26}
\int_K (-u)F_k[u] \leq \int_K (-R_k(K,\Omega))F_k[R_k(K,\Omega)]
\end{equation*}
holds for any $u\in \Phi_0^k(\Omega)\cap C^2(\bar{\Omega})$ with $-1<u<0$. As a consequence, we get
$$
\int_K (-R_k(K,\Omega))F_k[R_k(K,\Omega)]\ge cap_{k,4}(K,\Omega),
$$
thereby completing the argument.
\end{proof}

\begin{theorem}\label{t21} Suppose $1\leq k\leq \frac{n}{2}$. Let $\Omega$ be the Euclidean ball $B_r$ of radius $r$ centered at the origin. If $E\subset\Omega$, then
$$
cap_k(E,\Omega)=cap_{k,j}(E,\Omega)\ \forall\ j=1,2,3,4.
$$
\end{theorem}
\begin{proof} By Definition \ref{d21}, it is enough to prove that if $E=K$ is a compact subset of $\Omega$ then 
$$
cap_{k,1}(K,\Omega)\leq cap_{k,2}(K,\Omega)\leq cap_{k,3}(K,\Omega)\leq cap_{k,4}(K,\Omega)\leq cap_{k,1}(K,\Omega).
$$

To do so, note first that the inequalities 
$$
cap_{k,4}(K,\Omega)\leq cap_{k,1}(K,\Omega)\quad\&\quad
cap_{k,2}(K,\Omega)\leq cap_{k,3}(K,\Omega)
$$ 
just follow from Definition \ref{d21}. Next, an application of Lemma \ref{l21} yields
$$
cap_{k,1}(K,\Omega)=cap_k(K,\Omega)=\int_K F_k[R_k(K,\Omega)]=\int_{\Omega} F_k[R_k(K,\Omega)].
$$
Thus, from the definition of $R_k(K,\Omega)$ and the monotonicity described in the proof of Lemma \ref{l21}, it follows that for any $u\in \Phi_0^k(\Omega)\cap C^2(\bar{\Omega})$ satisfying $u|_K\leq -1$ and $u<0$ one has
$$
\int_{\Omega} F_k[R_k(K,\Omega)]\leq \int_{\Omega} F_k[u].
$$
Upon minimizing the right-hand side of the above inequality we obtain
$$
cap_{k,1}(K,\Omega)=\int_{\Omega} F_k[R_k(K,\Omega)]\leq cap_{k,2}(K,\Omega).
$$
Finally, by the definitions of $R_k(K,\Omega)$ and $cap_{k,3}(K,\Omega)$ we achieve
$$
cap_{k,3}(K,\Omega)\le \int_{\Omega} (-R_k(K,\Omega))F_k[R_k(K,\Omega)]= \int_K (-R_k(K,\Omega))F_k[R_k(K,\Omega)],
$$
thereby finding 
$$
cap_{k,3}(K,\Omega)\leq cap_{k,4}(K,\Omega).
$$ 
\end{proof}

\begin{corollary}\label{c21} Let $\Omega$ be the Euclidean ball $B_r$ of radius $r$ centered at the origin. If $E\subset\Omega$, then
$$
cap_1(E,\Omega)=\inf\Big\{\int_{\Omega}|Du|^2:\ u\in W^{1,2}(\Omega),\ u \geq 1_E\}\equiv 2\hbox{-}cap(E,\Omega),
$$
where $W^{1,2}(\Omega)$ stands for the Sobolev space of all functions whose distributional derivatives are in $L^2(\Omega)$.
\end{corollary}

\begin{proof} Thanks to the well-known metric properties of the Wiener capacity $2\hbox{-}cap(\cdot,\Omega)$ (cf. \cite[Chapter 2]{Maz0}), we only need to check that
$$
cap_1(E,\Omega)=2\hbox{-}cap(E,\Omega)\ \forall\ \hbox{compact } E\subset\Omega.
$$
Since $F_1[u]=\Delta u$, for any $u\in \Phi_0^k(\Omega)\cap C^2(\bar{\Omega})$ with $u\leq -1_E$ we apply an integration-by-part to obtain
$$
\int_{\Omega}(-u)F_1[u]=\int_{\Omega}(-u)\Delta u=\int_{\Omega}|Du|^2=\int_{\Omega}|D(-u)|^2.
$$
Upon considering the unique solution of the Dirichlet problem:
$$
\begin{cases}
F_1[u]=\Delta u=0 &\mbox{in } \Omega \backslash E; \\
-u=1 & \mbox{on } \partial E;\\
u=0 & \mbox{on } \partial \Omega,
\end{cases}
$$
we get
$$
cap_{1,3}(E,\Omega)=\int_{\Omega}(-R_1(E,\Omega))F_1[R_1(E,\Omega)]=\int_{\Omega}|D(-R_1(E,\Omega))|^2=2\hbox{-}cap(E,\Omega),
$$
whence reaching the conclusion via Theorem \ref{t21}.
\end{proof}

\section{Isocapacitary inequalities}\label{s3}

\subsection{Proof of Theorem \ref{t12} (i)}\label{s31} 

{\it Step $(i)_1$}. We start with proving that if $E\subset B_r$ and $1\leq k< \frac{n}{2}$ then there is a constant $c(n,k,q,|\Omega|)>0$ depending only on $n,k,q$, and $|\Omega|$ such that
\begin{equation*}
\label{e31}
|E|^\frac{k+1}{q}\leq c(n,k,q,|\Omega|) \big(cap_k(E, B_r)\big).
\end{equation*}

Without loss of generality, we may assume that $E$ is a compact set in $B_r$. Now, by Theorem \ref{t11} (i) we have that if $1\leq q\leq k^*$ then 
$$
\|u\|_{L^{q}(B_r)}\leq c(n,k,q,r)\|u\|_{\Phi_0^k(B_r)}\ \forall\ u\in \Phi_0^k(B_r),
$$
where $c(n,k,q,r)>0$ is a constant depending only on $n,k,q,r$.

Since $R_k(E,B_r)\in \Phi_0^k(B_r)$, from the definition of $\|\cdot\|_{\Phi_0^k(B_r)}$ it follows that
$$ 
\|R_k(E,B_r)\|_{L^q(B_r)}\leq c(n,k,q,r)\left(\int_{B_r} \big(-R_k(E,B_r)\big)F_k[R_k(E,B_r)]\right)^{\frac1{k+1}}.
$$
In other words, Theorem \ref{t21} is employed to derive 
$$
\|R_k(E,B_r)\|_{L^q(B_r)}\leq c(n,k,q,r) \Big(cap_k(E,B_r)\Big)^{\frac1{k+1}}.
$$
Thus, by the definition of $R_k(E,B_r)$ we achieve
\begin{eqnarray*}
&&|E|^\frac{k+1}{q}\\
&&\le \Big(\int_E |R_k(E,B_r)|^q \Big)^{\frac{k+1}{q}}\\
&&\le \Big(\int_{B_r} |R_k(E,B_r)|^q \Big)^{\frac{k+1}{q}}\\
&&\le \|R_k(E,B_r)\|_{L^q(B_r)}^{k+1}\\
&&\le\big(c(n,k,q,r)\big)^{k+1} cap_k(E,B_r).
\end{eqnarray*}

{\it Step $(i)_2$}. Next, we verify that if $E\subset\Omega$ and $1\leq k< \frac{n}{2}$ then there is a constant $c(n,k,q,|\Omega|)>0$ depending only on $n,k,q$ and $|\Omega|$ such that
\begin{equation*}
\label{e32}
|E|^\frac{k+1}{q}\leq c(n,k,q,|\Omega|)\big(cap_k(E,\Omega)\big)^n.
\end{equation*}

Without loss of generality, we may assume that $E$ is a compact subset of $\Omega$ containing the origin. Then there exists a ball $B_r$ centered at the origin with radius $r=\hbox{diam}(\Omega)$ such that $\Omega\subset B_r$.

Since $1\leq k< \frac{n}{2}$, by {\it Step $(i)_1$} and \cite[Lemma 4.1(ii)]{Lab} we obtain
$$
|E|^\frac{k+1}{q}\leq c(n,k,q,r)cap_k(E,B_r)\leq c(n,k,q,|\Omega|)cap_k(E,\Omega),
$$
as desired.

{\it Step $(i)_3$.} Particularly, for $q=\frac{n(k+1)}{n-2k}$ we make the following analysis. Suppose $E$ is a compact set contained in $B_r$ - a ball centered at the origin with radius $r>0$. We claim that if $1\leq k< \frac{n}{2}$ then there is a constant $c(n,k)>0$ depending only on $n,k$ such that
\begin{equation*}
\label{e311}
|E|^\frac{n-2k}{n}\leq c(n,k) cap_k(E, \mathbb{R}^n).
\end{equation*}
In fact, according to Dai-Bao's paper \cite{DaB} there exists a unique viscosity solution to the Dirichlet problem stated in the proof of Lemma \ref{l21}. Such a solution guarantees that there exists a unique $R_k(E,\mathbb{R}^n)$ satisfying
$$
R_k(E,\mathbb{R}^n)=\lim_{r\to \infty} R_k(E,B_r).
$$
Now, by the previous {\it Step $(i)_1$} we have that if $q=k^*$ then
$$
|E|^{\frac{n-2k}{n}}\leq c(n,k,r) cap_k(E,B_r),
$$
whence reaching the above claim through letting $r\to \infty$ in the above estimate.

Now, using the same argument for {\it Step $(i)_2$} we get
$$
|E|^\frac{n-2k}{n}\leq c(n,k) cap_k(E,\mathbb{R}^n) \leq c(n,k)cap_k(E,\Omega).
$$

{\it Step $(i)_4$.}  Following the above argument and applying \cite[Lemma 4.1(ii)]{Lab}, Theorem \ref{t11} (ii) and Theorem \ref{t21} we can get that
$$
|E|^\frac{k+1}{q}\leq c(n,k,q,\textnormal{diam}(\Omega))cap_k(E, \Omega)
$$
holds for $k=\frac{n}2$ and $1\le q<\infty$.

\subsection{Proof of Theorem \ref{t12} (ii)}\label{s32}

{\it Step $(ii)_1$}. Partially motivated by \cite{A, CL,XZ}, we begin with a slight improvement of the Moser-Trudinger inequality stated in Theorem \ref{t11} (ii): if  $k=\frac{n}{2}$ then there is a constant $c(n)>0$ depending only on $n$ such that

\begin{equation*}
	\label{e16ee}
\sup_{0<\|u\|_{\Phi_0^k(\Omega)}<\infty}\int_\Omega\exp\left(\alpha\Big(\frac{|u|}{\|u\|_{\Phi_0^k(\Omega)}}\Big)^\beta\right)\le c(n)\big(\hbox{diam}(\Omega)\big)^n,
	\end{equation*}
	where $\alpha,\beta$ are the constants determined in Theorem \ref{t11} (ii).
	
Without loss of generality, we may assume that $\Omega$ contains the origin. Then there exists a ball $B_r$ centered at the origin with radius $\hbox{diam}(\Omega)$ such that $\Omega\subset B_r$.
Following the argument for \cite[Theorem 1.2]{TiW}, we have that, for any radial function $u=u(s)$ in $\Phi_0^k(B_r)$, there exists a ball $B_{\hat{r}}\subset \mathbb{R}^{\frac{n}{2}+1}$ with radius $\hat{r}=r^{\frac{2n}{n+2}}$ and a radial function $v(s)=u(s^{\frac{n+2}{2n}})$ in $\Phi_0^k(B_{\hat{r}})$ such that
\begin{eqnarray*}
&&\int_{\Omega}\exp\left(\alpha\big(\frac{|u|}{\|u\|_{\Phi_0^k(B_r)}}\big)^{\beta}\right)\\
&&\le \Big(\frac{n+2}{2n}\Big)\Big(\frac{\omega_{n-1}}{\omega_{\frac{n}{2}}}\Big)\int_{B_{\hat{r}}}\exp\left(\frac{\alpha}{c_0^{\beta}}\Big( \frac{|v|}{\|Dv\|_{L^{\frac{n}{2}+1}(B_{\hat{r}})}}\Big) \right)\\
&&\le c(n) |B_{\hat{r}}|\\
&&\le c(n) {\hat{r}}^{\frac{n}{2}+1}\\
&&\le c(n) r^n,
\end{eqnarray*}
where 
$$
c_0^{\beta}=\left(\frac{\omega_{n-1}}{k\omega_{n/2}}\small{\begin{pmatrix} n-1\\k-1  \end{pmatrix}}\big(\frac{2n}{n+2}\big)^{\frac{n}{2}} \right)^{\frac{1}{k+1}}.
$$
Thus, by \cite[Lemma 3.2]{TiW} we achieve
\begin{eqnarray*}
&&\sup\left\{ \int_{\Omega}\exp\Big(\alpha\big(\frac{|u|}{\|u\|_{\Phi_0^k(\Omega)}}\big)^{\beta}\Big):\ u\in \Phi_0^k(\Omega)\ \&\ 0<\|u\|_{\Phi_0^k(\Omega)}<\infty\right\}\\
&&\le \sup\left\{ \int_{\Omega}\exp\Big(\alpha\big(\frac{|u|}{\|u\|_{\Phi_0^k(\Omega)}}\big)^{\beta}\Big):\ u\in \Phi_0^k(\Omega)\ \textnormal{is\ radial}\right\}\\
&&\le c(n)\big(\textnormal{diam}(\Omega)\big)^n,
\end{eqnarray*}
as desired.

{\it Step $(ii)_2$}. We utilize the last step to check the remaining part of Theorem \ref{t12} (ii). Since $k=\frac{n}{2}$, by Lemmas \ref{l21}\&\ref{e16ee} and Theorem \ref{t21} we have
\begin{eqnarray*}
&&|E|\exp\Big(\frac{\alpha}{\big(cap_k(E,B_r)\big)^{\frac{\beta}{k+1}}}\Big)\\
&&=|E|\exp\Big(\frac{\alpha}{\big(cap_{k,3}(E,B_r)\big)^{\frac{\beta}{k+1}}}\Big)\\
&&\le\sup\Big\{\int_E \exp\Big(\alpha\big( \frac{|u|}{\|u\|_{\Phi_0^k(B_r)}}\big)^{\beta} \Big):\ u\in \Phi_0^k(B_r)\Big\}\\
&&\le c(n)\big(\hbox{diam}(B_r)\big)^n,
\end{eqnarray*}
i.e.,
$$
\frac{\alpha}{\big(cap_k(E,\Omega)\big)^{\frac{\beta}{k+1}}}\leq \frac{\alpha}{\big(cap_k(E,B_r)\big)^{\frac{\beta}{k+1}}}\leq \ln\Big( {c(n)|E|^{-1}\big(\hbox{diam}(\Omega)\big)^n}\Big).
$$
Now, a simple calculation gives the desired inequality.

\section{Capacitary weak and strong type estimates for $\Phi_0^k(\Omega)$}\label{s4}

Through introducing a new measure-theoretic method that is evidently different from the usual dyadic-splitting-level-set approach to the capacitary weak and strong type estimates for the Wiener capacity $2\hbox{-}cap(\cdot,\Omega)=cap_1(\cdot,\Omega)$ (with $\Omega$ being an origin-centered Euclidean ball) -- see also Adams-Hedberg's \cite[Chapter 7]{AH} and the references therein for an account on the classical treatment of some relevant capacity estimates, we establish the following $k$-Hessian capacitary weak and strong type inequalities.

\begin{theorem}\label{t41} Suppose that $\Omega$ is an origin-centered Euclidean ball. If $u\in \Phi_0^k(\Omega)\cap C^2(\bar{\Omega})$ and $1\le k\le\frac{n}{2}$, then one has:

\begin{enumerate}[\rm(i)]
\item the capacitary weak type inequality
$$
cap_k\big(\{x\in \Omega:\ |u(x)|\ge t\},\Omega\big)\leq {t^{-(k+1)}}\|u\|_{\Phi_0^k(\Omega)}^{k+1}\ \ \forall\ \ t>0;
$$
\item the capacitary strong type inequality
$$
\int_0^{\infty} t^k cap_k\big(\{x\in \Omega: |u(x)|\ge t\},\Omega\big) \,dt\leq \Big(\frac{a}{a-1}\Big)^{k+1}(\ln a) \|u\|_{\Phi_0^k(\Omega)}^{k+1}\ \ \forall\ \ a>1.
$$
\end{enumerate}
\end{theorem}
\begin{proof} (i) For $t>0$ let $v=t^{-1}{u}$. By Theorem \ref{t21} we obtain
\begin{eqnarray*}
&&cap_k\big(\{ x\in \Omega: |v(x)|\ge 1\},\Omega\big)\\
&&=\sup\left\{\int_{\{|v|\ge 1\}}(-f)F_k[f]:\ \ 
f\in\Phi_0^k(\Omega)\cap C^2(\bar{\Omega}),\ -1<f<0\right\}\\
&&= \int_{\{|v|\ge 1\}}(-R(\{|v|\ge 1\},\Omega))F_k[R(\{|v|\ge 1\},\Omega)]\\
&&\le \int_{\Omega}(-R(\{|v|\ge 1\},\Omega))F_k[R(\{|v|\ge 1\},\Omega)]\\
&&\le \int_{\Omega}(-v)F_k[R(\{|v|\ge 1\},\Omega)]\\
&&\le \int_{\Omega}(-v)F_k[v],
\end{eqnarray*}
thereby getting
$$
cap_k\big(\{x\in \Omega: |u(x)|\ge t\},\Omega\big)\leq {t^{-(k+1)}}\int_{\Omega}(-u)F_k[u].
$$

(ii) For $t>0$ let 
$$
M_t(u)=\{x\in\Omega: |u(x)|\ge t\}\quad\&\quad \Omega_t(u)=\{x\in\Omega: |u(x)|>t\}.
$$
Without loss of generality, we may assume $\|u\|_{\Phi_0^k(\Omega)}<\infty$, and then define a normalized set function (cf. \cite[Theorem 2.2-Corollary 2.3]{KM}-based argument for \cite[Theorem 3.1]{CMS})
$$
\phi(E)\equiv\phi(E,\Omega)=\frac{\int_E (-u)F_k[u]}{\|u\|^{k+1}_{\Phi_0^k(\Omega)}}\quad\forall\quad E\subset\Omega.
$$
Note that
$$
\begin{cases}
\phi(\emptyset)=0\quad\&\quad\phi(E)\ge 0\ \forall\ E\subset\Omega;\\
\phi(E_1)\le\phi(E_2)\ \forall\ E_1\subset E_2\subset\Omega;\\
\phi(E_1\cup E_2)=\phi(E_1)+\phi(E_2)\ \forall\ E_1, E_2\subset\Omega\ \hbox{ with }\ E_1\cap E_2=\emptyset.
\end{cases}
$$
So, $\phi$ is a measure. Consequently, for a given constant $a>1$ we estimate
\begin{eqnarray*}
&&\int_0^{\infty}\phi(\Omega_t(u)\backslash M_{at}(u))\,\frac{dt}{t}\\
&&\le\int_0^{\infty}\phi(M_t(u)\backslash M_{at}(u))\,\frac{dt}{t}\\
&&=\int_0^{\infty} \int_t^{at} d\phi(M_s(u))\, \frac{dt}{t}\\
&&= \int_0^{\infty} \int_s^{\frac{s}{a}} \frac{dt}{t}\, d\phi(M_s(u))\\
&&=-(\ln a) \int_0^{\infty}\, d\phi(M_s(u))\\
&&= \phi(M_0(u))\ln a\\
&&\le\phi(\Omega)\ln a\\
&&=\ln a,
\end{eqnarray*}
whence deriving
$$
\int_0^{\infty}\big\|u 1_{\Omega_t(u)\backslash M_{at}(u)}\big\|^{k+1}_{\Phi_0^k(\Omega)}\frac{dt}{t}=\int_{\Omega_t(u)\setminus M_{at}(u)}(-u)F_k[u]\leq \|u\|^{k+1}_{\Phi_0^k(\Omega)}\ln a.
$$
Now, if 
$$
\tilde{u}=\max\Big\{\frac{t-u}{(a-1)t},-1\Big\},
$$
then 
$$
\begin{cases}
\tilde{u}\in \Phi_0^k(\Omega_t(u));\\
\tilde{u}1_{M_{at}}\leq -1,
\end{cases}
$$
and hence
\begin{eqnarray*}
&&\|\tilde{u}\|^{k+1}_{\Phi_0^k(\Omega_t(u))}\\
&&=\int_{\Omega_t(u)} (-\tilde{u})F_k[\tilde{u}]\\
&&={k}^{-1}\int_{\Omega_t(u)} \tilde{u}_i\tilde{u}_j F_k^{ij}[D^2\tilde{u}]\\
&&={k}^{-1}\int_{\Omega_t(u)\backslash M_{at}(u)}\left(\frac{u}{(a-1)t}\right)_i \left(\frac{u}{(a-1)t}\right)_j F_k^{ij}\left[D^2 \frac{u}{(a-1)t}\right]\\
&&\le\int_{\Omega_t(u)\backslash M_{at}(u)}\left(-\frac{u}{(a-1)t}\right)F_k\left[\frac{u}{(a-1)t}\right]\\
&&= {(a-1)^{-k-1}t^{-k-1}}\int_{\Omega_t\backslash M_{at}}(-u)F_k[u],
\end{eqnarray*}
where 
$$
\begin{cases}
F_k^{ij}[A]=\frac{\partial}{\partial a_{ij}}F_k[A];\\
D^2 f=A=\{a_{ij}\}.
\end{cases}
$$
Using the definition of $cap_{k,3}(\cdot,\Omega)$, we obtain
\begin{eqnarray*}
&&
\int_0^{\infty}t^{k+1}cap_{k,3}(M_{at}(u),\Omega_t(u))\,\frac{dt}{t}\\
&&\le \int_0^{\infty}t^{k+1}\|\tilde{u}\|^{k+1}_{\Phi_0^k(\Omega_t(u))}\,\frac{dt}{t}\\
&&\le \int_0^{\infty}{(a-1)^{-(k+1)}}\Big(\int_{\Omega_t(u)\backslash M_{at}(u)}(-u)F_k[u]\Big)\,\frac{dt}{t}\\
&&\le{(\ln a)}{(a-1)^{-(k+1)}}\|u\|^{k+1}_{\Phi_0^k(\Omega)}.
\end{eqnarray*}
This last inequality, along with $\lambda=at$ and Theorem \ref{t21}, derives
\begin{align*}
&\int_0^{\infty} \lambda^k cap_k\big(M_\lambda(u),\Omega\big) \,d\lambda\\
&\le \int_0^{\infty} (at)^k cap_{k,3} (M_{at}(u),\Omega_t(u)) \,d(at)\\
&\le{\Big(\frac{a}{a-1}\Big)^{k+1}(\ln a)}\|u\|^{k+1}_{\Phi_0^k(\Omega)}.
\end{align*}
\end{proof}

\begin{remark}
\label{r41}
Interestingly, Theorem \ref{t41} (i) with $t>1$ may be regarded as the bottom of a consequence of Theorem \ref{t41} (ii). In fact, since $t\mapsto cap_k\big(M_t(u),\Omega\big)$ is a decreasing function, it follows that 
$$
\Big(\frac{a^{k+1}}{k+1}\Big)cap_k\big(M_a(u),\Omega\big)\le
\int_0^{a} t^k cap_k\big(M_t(u),\Omega\big) \,dt\leq \Big(\frac{a}{a-1}\Big)^{k+1}(\ln a) \|u\|_{\Phi_0^k(\Omega)}^{k+1}\ \forall\ a>1,
$$
and consequently,
$$
cap_k\big(M_a(u),\Omega\big)\le\frac{(k+1)\ln a}{(a-1)^{k+1}}\|u\|_{\Phi_0^k(\Omega)}^{k+1}\ge a^{-(k+1)}\|u\|_{\Phi_0^k(\Omega)}^{k+1}\ \forall\ a>1.
$$
\end{remark}

\section{Analytic vs geometric trace inequalities}\label{s5}

\subsection{Proof of Theorem \ref{t13} (i)}\label{s51} In what follows, we always let
$$
\begin{cases}
1\leq k\leq \frac{n}{2};\\
u\in \Phi_0^k(\Omega)\cap C^2(\bar{\Omega});\\
M_{t}(u)=\{x\in\Omega:\ |u(x)|\geq t\}\quad\forall\quad t>0.
\end{cases}
$$
{\it Step $(i)_1$.} For $k+1\leq q<\infty$ let
$$
C_1\equiv\sup_{t>0}\frac{t^{\frac{k+1}{q}}}{\tau(\mu,\Omega,t)}<\infty.
$$
Then
$$
\mu(K)^{\frac{1}{q}}\leq C_1 ^\frac{1}{k+1} \big(cap_k(K,\Omega)\big)^{\frac{1}{k+1}}\quad\forall\quad \hbox{compact}\ K\subset \Omega.
$$
An application of Theorem \ref{t41} (ii) with $a=n$ yields that for any $u\in \Phi_0^k(\Omega)\cap C^2(\bar{\Omega})$ one has
\begin{eqnarray*}
&&\int_{\Omega}|u|^q\, d\mu\\
&&=\int_0^{\infty}\mu(M_{\lambda}(u))\,d\lambda^q\\
&&\le C_1^\frac{q}{k+1}\int_0^{\infty}\Big(cap_k(M_{\lambda}(u),\Omega)\Big)^{\frac{q}{k+1}}\,d\lambda^q\\
&&\le q(k+1)^{-1}C_1^\frac{q}{k+1}\|u\|_{\Phi_0^k(\Omega)}^{q-k-1}\int_0^{\infty} cap_k(M_{\lambda}(u),\Omega)\,d\lambda^{k+1}\\
&&\le q(k+1)^{-1}C_1^\frac{q}{k+1}c(n,k)\|u\|_{\Phi_0^k(\Omega)}^q.
\end{eqnarray*}
This gives
$$
C_2\equiv\sup\left\{\frac{\|u\|_{L^q(\Omega,\mu)}}{\|u\|_{\Phi_0^k(\Omega)}}:\ u\in \Phi_0^k(\Omega)\cap C^2(\bar{\Omega})\ \hbox{with}\ 0<\|u\|_{\Phi_0^k(\Omega)}<\infty\right\}<\infty.
$$ 

Conversely, assume $C_2<\infty$. An application of the H\"older inequality with $q'=\frac{q}{q-1}$ implies
\begin{eqnarray*}
&&t\mu(M_{t}(u))\\
&&\le\int_{\Omega}|u|1_{M_t(u)}\,d\mu\\
&&\le\|u\|_{L_q(\Omega,\mu)}\big(\mu(M_{t}(u))\big)^{\frac{1}{q'}}\\
&&\le C_2\|u\|_{\Phi_0^k(\Omega)}\big(\mu(M_{t}(u))\big)^{\frac{1}{q'}},
\end{eqnarray*}
and thus
$$
\sup_{t>0}t \big(\mu(M_{t}(u))\big)^\frac1q\leq C_2\|u\|_{\Phi_0^k(\Omega)}.
$$
Now, taking 
$$
\begin{cases}
t=1;\\
u\in \Phi_0^k(\Omega)\cap C^2(\bar{\Omega});\\
|u|\geq 1_K\ \hbox{for\ any\ compact}\ K\subset\Omega,
\end{cases}
$$ 
we obtain
$$
\big(\mu(K)\big)^{\frac{1}{q}}\leq C_2\|u\|_{\Phi_0^k(\Omega)} \leq C_2 \big(cap_k(K,\Omega)\big)^{\frac{1}{k+1}},
$$
whence reaching $C_1\le C_2^{k+1}$.

{\it Step $(i)_2$.} For $1<q<k+1$ let
$$
\begin{cases}
I_{k,q}(\mu)\equiv\int_0^{\infty}\Big(t^{\frac{k+1}{q}}\big(\tau(\mu,\Omega,t)\big)^{-1}\Big)^{\frac{q}{k+1-q}}t^{-1}\, dt;\\
S_{k,q}(\mu,u)\equiv\sum_{j=-\infty}^{\infty}\frac{\big(\mu(M_{2^j}(u))-\mu(M_{2^{j+1}}(u))\big)^{\frac{k+1}{k+1-q}}}{\big(cap_k(M_{2^j}(u))\big)^{\frac{q}{k+1-q}}}.
\end{cases}
$$

If $I_{k,q}(\mu)<\infty$, then the elementary inequality 
$$
a^c+b^c\le(a+b)^c\ \ \forall\ \ a,b\ge 0\ \&\ c\ge 1
$$
implies
\begin{eqnarray*}
&&S_{k,q}(\mu,u)\\
&&=\sum_{j=-\infty}^{\infty}{\big(\mu(M_{2^j}(u))-\mu(M_{2^{j+1}}(u))\big)^{\frac{k+1}{k+1-q}}}{\big(cap_k(M_{2^j}(u),\Omega)\big)^{-\frac{q}{k+1-q}}}\\
&&\le \sum_{j=-\infty}^{\infty}{\big(\mu(M_{2^j}(u))-\mu(M_{2^{j+1}}(u))\big)^{\frac{k+1}{k+1-q}}}{\big(\tau(\mu,\Omega,\mu(M_{2^j}(u)))\big)^{-\frac{q}{k+1-q}}}\\
&&\le  \sum_{j=-\infty}^{\infty}{\mu(M_{2^j}(u))^{\frac{k+1}{k+1-q}}-\mu(M_{2^{j+1}}(u))^{\frac{k+1}{k+1-q}}}{\big(\tau(\mu,\Omega,\mu(M_{2^j}(u)))\big)^{-\frac{q}{k+1-q}}}\\
&&\le c(n,k,q) \int_0^{\infty}{(\tau(\mu,\Omega,s))^{-\frac{q}{k+1-q}}}\,ds^{\frac{k+1}{k+1-q}}\\
&&\le c(n,k,q) I_{k,q}(\mu).
\end{eqnarray*}
Therefore, by the H\"older inequality and Theorem \ref{t41} (ii) with $a=n$ we have
\begin{eqnarray*}
&&\|u\|^q_{L^q(\Omega,\mu)}\\
&&= \int_{0}^\infty\mu\big(M_t(u)\big)\,dt^q\\
&&=-\int_0^{\infty}t^q d\mu(M_{t}(u))\\
&&\le \sum_{-\infty}^{\infty} \big(\mu(M_{2^j}(u))-\mu(M_{2^{j+1}}(u))\big)2^{(j+1)q}\\
&&\le (S_{k,q}(\mu,u))^{\frac{k+1-q}{k+1}} \left( \sum_{-\infty}^{\infty} 2^{j(k+1)} cap_k(M_{2^{j(k+1)}}(u),\Omega)\right)^{\frac{q}{k+1}}\\
&&\le (S_{k,q}(\mu,u))^{\frac{k+1-q}{k+1}} \left( \int_0^{\infty} cap_k(M_{\lambda}(u),\Omega)\,d\lambda^{k+1}\right)^{\frac{q}{k+1}}\\
&&\le c(n,k,q) (S_{k,q}(\mu,u))^{\frac{k+1-q}{k+1}}\|u\|^q_{\Phi_0^k(\Omega)}\\
&&\le c(n,k,q) (I_{k,q}(\mu))^{\frac{k+1-q}{k+1}}\|u\|^q_{\Phi_0^k(\Omega)},
\end{eqnarray*}
whence getting 
$$
C_2^q\le c(n,k,q)\big(I_{k,q}(\mu)\big)^{\frac{k+1-q}{k+1}}.
$$

Conversely, suppose $C_2<\infty$. Then
$$
\sup_{t>0}t\big(\mu(M_{t})\big)^{\frac{1}{q}}\le\|u\|_{L^q(\Omega,\mu)}\leq C_2 \|u\|_{\Phi_0^k(\Omega)}
$$
holds for any $u\in \Phi_0^k(\Omega)\cap C^2(\bar{\Omega})$. According to the definition of $\tau(\mu,\Omega,t)$, for each integer $j$ there exist a compact set $K_j\subset \Omega$ and a function $u_j\in \Phi_0^k(\Omega)\cap C^2(\bar{\Omega})$ such that
$$
\begin{cases}
cap_k(K_j,\Omega)\leq 2\tau(\mu,\Omega, 2^j); \\
\mu(K_j)>2^j;\\
u_j\leq -1_{K_j};\\
2^{-1}\|u_j\|_{\Phi_0^k(\Omega)}^{k+1}\leq cap_k(K_j,\Omega).
\end{cases}
$$
Now, for integers $i$, $m$ with $i<m$ let
$$
\begin{cases}
u_{i,m}=\sup_{i\leq j\leq m}\gamma_j f_j;\\
\gamma_j=\Big(\frac{2^j}{\tau(\mu,\Omega,2^j)}\Big)^{\frac{1}{k+1-q}}\ \forall\ j\in [i,m].
\end{cases}
$$
Then $u_{i,m}$ is a function in $\Phi_0^k(\Omega)\cap C^2(\bar{\Omega})$ -- this follows from an induction and the easily-checked fact below
$$
\max\{u_1,u_2\}=\frac{u_1+u_2+|u_1-u_2|}{2}\in \Phi_0^k(\Omega)\cap C^2(\bar{\Omega}).
$$ 
Consequently,
$$
\|u_{i,m}\|_{\Phi_0^k(\Omega)}^{k+1}\leq c(n,k) \sum_{j=i}^m \gamma_j^{k+1}\|u_j\|_{\Phi_0^k(\Omega)}^{k+1}\leq c(n,k) \sum_{j=i}^m \gamma_j^{k+1}\tau(\mu,\Omega, 2^j).
$$
Observe that for $i\leq j\leq m$ one has
$$
u_{i,m}(x)\leq \gamma_j\ \forall\  x\in K_j.
$$
Therefore,
$$
2^j<\mu(K_j)\leq \mu\left(M_{\gamma_j}(u_{i,m})\right).
$$
This in turn implies
\begin{eqnarray*}
&&\|u_{i,m}\|_{\Phi_0^k(\Omega)}^q\\
&&\ge C_2^{-q} c(n,k,q)\int_{\Omega}|u_{j,m}|^q\,d\mu\\
&&\ge C_2^{-q} \int_0^{\infty}\Big(\inf\{t:\ \mu(M_{t}(u_{i,m}))\leq s\}\Big)^q\, ds\\
&&\ge C_2^{-q} \sum_{j=i}^m \Big(\inf\{t:\ \mu(M_{t}(u_{i,m}))\leq 2^j\}\Big)^q 2^j\\
&&\ge C_2^{-q} \sum_{j=i}^m \gamma_j^q 2^j\\
&&\ge C_2^{-q}c(n,k,q)\left( \frac{\sum_{j=i}^m\gamma_j^q 2^j}{\Big(\sum_{j=i}^m   \big(\gamma_j\big)^{k+1}\tau(\mu,\Omega, 2^j)\Big)^{\frac{q}{k+1}}}\right)\|u_{i,m}\|_{\Phi_0^k(\Omega)}^q\\
&&\ge C_2^{-q}c(n,k,q) \left(\frac{\sum_{j=i}^m {2^{\frac{j(k+1)}{k+1-q}}}{\big( \tau(\mu,\Omega, 2^j)\big)^{-\frac{q}{k+1-q}}}}{\Big(\sum_{j=i}^m 2^{\frac{j(k+1)}{k+1-q}} \big(\tau(\mu,\Omega, 2^j)\big)^{\frac{-q}{k+1-q}}\Big)^{\frac{q}{k+1}}}\right)\|u_{i,m}\|_{\Phi_0^k(\Omega)}^q\\
&&\ge C_2^{-q} c(n,k,q) \left(\sum_{j=i}^m{2^{\frac{j(k+1)}{k+1-q}}}{\big(\tau(\mu,\Omega, 2^j)\big)^{-\frac{q}{k+1-q}}}\right)^{\frac{k+1-q}{k+1}}\|u_{i,m}\|_{\Phi_0^k(\Omega)}^q.
\end{eqnarray*}
Consequently,
$$
I_{k,q}(\mu)\leq \lim_{i\to-\infty,\
 m\to\infty } \sum_{j=i}^m{2^{\frac{(j+1)(k+1)}{k+1-q}}}{(\tau(\mu,\Omega, 2^j))^{-\frac{q}{k+1-q}}}<\infty.
$$

\subsection{Proof of Theorem \ref{t13} (ii)}\label{s53} In the sequel, set
$$
\begin{cases}
k=\frac{n}{2};\\
u\in \Phi_0^k(\Omega)\cap C^2(\bar{\Omega});\\
M_{t}(u)=\{x\in \Omega:\ |u(x)|\ge t\}\ \ \forall\ \ t>0.
\end{cases}
$$
For convenience, rewrite the previous quantity $C_1$ as
$$
C_1(n,k,q,\mu,\Omega)\equiv\sup_{t>0}\frac{t^{\frac{k+1}{q}}}{\tau(\mu,\Omega,t)}.
$$
If
$$
C_3(n,k,\alpha,\beta,\mu,\Omega)\equiv\sup_{t>0}t \exp\left(\frac{\alpha}{\big(\tau(\mu,\Omega,t)\big)^{\frac{\beta}{k+1}}}\right)<\infty,
$$
then for $\tilde{q}\geq k+1$ one has
\begin{eqnarray*}
&&C_1(n,k,\tilde{q},\mu,\Omega)\\
&& =\sup_{t>0}\frac{t^{\frac{k+1}{\tilde{q}}}}{\tau(\mu,\Omega,t)}\\
&& =\sup_{t>0} \left(\Big(\frac{\tilde{q}t^{\frac{\beta}{\tilde{q}}} }{\alpha\beta}\Big) \Big(\frac{\alpha{\beta}{\tilde{q}}^{-1}}{\big(\tau(\mu,\Omega,t)\big)^{\frac{\beta}{k+1}}}\Big)\right)^{\frac{k+1}{\beta}}\\
&& \le \left(\frac{\tilde{q}}{\alpha\beta}\right)^{\frac{k+1}{\beta}} \sup_{t>0} \left(t^{\frac{\beta}{\tilde{q}}}\exp\Big(\frac{\alpha{\beta}{\tilde{q}}^{-1}}{\big(\tau(\mu,\Omega,t)\big)^{\frac{\beta}{k+1}}}\Big)\right)^{\frac{k+1}{\beta}}\\
&& =\left(\frac{\tilde{q}}{\alpha\beta}\right)^{\frac{k+1}{\beta}} \sup_{t>0}\left(t\exp\Big(\frac{\alpha}{\big(\tau(\mu,\Omega,t)\big)^{\frac{\beta}{k+1}}}\Big)\right)^{\frac{k+1}{\tilde{q}}}\\
&& \le\left(\frac{\tilde{q}}{\alpha\beta}\right)^{\frac{k+1}{\beta}}\big(C_3(n,k,\mu,\Omega)\big)^{\frac{k+1}{\tilde{q}}}.
\end{eqnarray*}
Also, applying the H\"older inequality for $\tilde{q}\geq k+1$ we get
\begin{eqnarray*}
&&\int_{\Omega}\exp\Big(\alpha\big(\frac{|u|}{\|u\|_{\Phi_0^k(\Omega)}}\big)^{\beta}\Big)\,d\mu\\
&&=\sum_{i=1}^{\infty}\int_{\Omega}\frac{\alpha^i}{i!}\big(\frac{|u|}{\|u\|_{\Phi_0^k(\Omega)}}\big)^{\beta i}\,d\mu\\
&&=\sum_{i<\frac{k+1}{\beta}}\int_{\Omega}\frac{\alpha^i}{i!}\big(\frac{|u|}{\|u\|_{\Phi_0^k(\Omega)}}\big)^{\beta i}\,d \mu+ \sum_{i\geq\frac{k+1}{\beta}}\int_{\Omega}\frac{\alpha^i}{i!}\big(\frac{|u|}{\|u\|_{\Phi_0^k(\Omega)}}\big)^{\beta i}\,d \mu\\
&&\le S_1+S_2,
\end{eqnarray*}
where
$$
\begin{cases}
S_1\equiv\sum_{i<\frac{k+1}{\beta}}\frac{\alpha^i}{i!}\big(\mu(\Omega)\big)^{1-\frac{\beta i}{\tilde{q}}}\Big(\int_{\Omega}\big(\frac{|u|}{\|u\|_{\Phi_0^k(\Omega)}}\big)^{\tilde{q}}\,d \mu\Big)^{\frac{\beta i}{\tilde{q}}};\\
S_2\equiv\sum_{i\geq\frac{k+1}{\beta}}\frac{\alpha^i}{i!}\int_{\Omega}\big(\frac{|u|}{\|u\|_{\Phi_0^k(\Omega)}}\big)^{\beta_0 i}\,d \mu.
\end{cases}
$$

Next, we control $S_1$ and $S_2$ from above. As estimated in the previous subsection, for any $u\in \Phi_0^k(\Omega)\cap C^2(\bar{\Omega})$ and integer $m\geq k+1$ we have
$$
\int_{\Omega}|u|^m \,d\mu\leq \big(C_1(n,k,m,\mu,\Omega)\big)^{\frac{m}{k+1}}c(n,k)\|u\|_{\Phi_0^k(\Omega)}^m.
$$
This, along with the previously-verified inequality
$$
C_1(n,k,\tilde{q},\mu,\Omega)\le \left(\frac{\tilde{q}}{\alpha\beta}\right)^{\frac{k+1}{\beta}} \big(C_3(n,k,\mu,\Omega)\big)^\frac{k+1}{\tilde{q}}\ \forall\ \tilde{q}\ge k+1,
$$ 
gives
$$
S_1\leq \sum_{i<\frac{k+1}{\beta}}\frac{\alpha^i}{i!}\big(\mu(\Omega)\big)^{1-\frac{\beta i}{\tilde{q}}}\Big(\big(C_1(n,k,\tilde{q},\mu,\Omega)\big)^{\frac{\tilde{q}}{k+1}}c(n,k)\Big)^{\frac{\beta i}{\tilde{q}}}<\infty.
$$
Meanwhile, Theorem \ref{t41} (ii) with $a=n$ is used to derive
\begin{eqnarray*}
&& S_2\\
&&=\sum_{i\geq\frac{k+1}{\beta}}\frac{\alpha^i}{i!}\|u\|_{\Phi_0^k(\Omega)}^{-\beta i}\int_{\Omega}|u|^{\beta i}\,d \mu\\
&& = \sum_{i\geq\frac{k+1}{\beta}}\frac{\alpha^i}{i!}{\|u\|_{\Phi_0^k(\Omega)}^{-\beta i}}\int_0^\infty \mu(M_t(u))\,dt^{\beta i}\\
&& = \sum_{i\geq\frac{k+1}{\beta}}\frac{\alpha^i}{i!}\int_0^\infty \frac{\big(cap_k(M_t(u),\Omega)\big)^\frac{\beta i}{k+1}}{\|u\|_{\Phi_0^k(\Omega)}^{\beta i}} \left(\frac{\mu(M_t(u))}{\big(cap_k(M_t(u),\Omega)\big)^\frac{\beta i}{k+1}}\right)\,dt^{\beta i}\\
&& \le \sum_{i\geq\frac{k+1}{\beta}}\frac{\alpha^i}{i!}\int_0^\infty \frac{cap_k(M_t(u),\Omega)}{t^{\beta i-k-1}}\left(\frac{\|u\|_{\Phi_0^k(\Omega)}^{\beta i-k-1}}{\|u\|_{\Phi_0^k(\Omega)}^{\beta i}}\right)\left( \frac{\mu(M_t(u))}{\big(cap_k(M_t(u),\Omega)\big)^\frac{\beta i}{k+1}}\right)\,dt^{\beta i}\\
&&\le\frac{\alpha\beta}{k+1}\int_0^\infty \sum_{i=0}^\infty \frac{\alpha^i}{i!}\left( \frac{\mu(M_t(u))}{\big(cap_k(M_t(u),\Omega)\big)^\frac{\beta i}{k+1}}\right)cap_k(M_t(u),\Omega) \|u\|_{\Phi_0^k(\Omega)}^{-(k+1)} \,dt^{k+1}\\
&&\le\frac{\alpha\beta}{k+1}\int_0^\infty \left(\mu(M_t(u))\exp\Big(\frac{\alpha}{\big(cap_k(M_t(u),\Omega)\big)^\frac{\beta}{k+1}}\Big)\right)\left(\frac{cap_k(M_t(u),\Omega)}{ \|u\|_{\Phi_0^k(\Omega)}^{k+1}}\right)\,dt^{k+1}\\
&&\le\alpha\beta(k+1)^{-1} C_3(n,k,\alpha,\beta,\mu,\Omega) \|u\|_{\Phi_0^k(\Omega)}^{-(k+1)} \int_0^\infty \big(cap_k(M_t(u),\Omega)\big)\,dt^{k+1}\\
&& \le \alpha\beta c(n,k)C_3(n,k,\alpha,\beta,\mu,\Omega).
\end{eqnarray*}
Now, putting the estimates for $S_1$ and $S_2$ together we obtain 
\begin{equation*}
\label{e1e3}
C_4\equiv\sup\left\{\|u\|_{L^1_\varphi(\Omega,\mu)}: \  u\in \Phi_0^k(\Omega)\cap C^2(\bar{\Omega})\ \hbox{with}\ \|u\|_{\Phi^k_0(\Omega)}>0\right\}<\infty.
\end{equation*}

Conversely, if $C_4<\infty$, then for any $u\in \Phi_0^k(\Omega)\cap C^2(\bar{\Omega})$ {with} $\|u\|_{\Phi^k_0(\Omega)}>0$ one always has
$$
\int_{\Omega}\exp\left(\alpha\big(\frac{|u|}{\|u\|_{\Phi_0^k(\Omega)}}\big)^{\beta}\right)\,d \mu\le C_4.
$$
Note that for any compact set $K\subset \Omega$ there exists a function $R_k(K,\Omega)$ such that
$$
\begin{cases}
R_k(K,\Omega)\in \Phi_0^k(\Omega)\cap C^2(\bar{\Omega});\\
|R_k(K,\Omega)|\geq 1_K.
\end{cases}
$$ 
So, we get
\begin{eqnarray*}
&&\mu(K) \exp\left(\frac{\alpha}{\big(cap_k(K,\Omega)\big)^\frac{\beta}{k+1}}\right)\\
&&\le \int_K \exp\left(\frac{\alpha}{\big(cap_k(K,\Omega)\big)^\frac{\beta}{k+1}}\right)\, d\mu\\
&&\le \int_\Omega \exp\left(\alpha \Big(\frac{|R_k(K,\Omega)|}{\|R_k(K,\Omega)\|_{\Phi_0^k(\Omega)}}\Big)^{\beta}\right)\, d\mu\\
&&\le C_4,
\end{eqnarray*}
whence achieving $C_3(n,k,\alpha,\beta,\mu,\Omega)\le C_4$.

\begin{proof}[\bf Acknowledgements] We would like to thank the referee for his/her remarks that indicate a possibility of utilizing both the equivalence between the $k$-Hessian capacity and the Bessel capacity found in \cite[Theorem 20]{PV} and the capacitary weak-strong inequalities explored in Adams-Hedberg's \cite[Chapter 7]{AH} to obtain Theorem \ref{t13} (i) under $1\le k<n/2$ (see \cite[p.211, 7.6.6]{AH} or directly \cite[Section 8.4.2]{Maz0}, \cite[Appendix, Theorem 2]{MazN} and \cite{Ve}). 
\end{proof}

\end{document}